\documentclass{amsart}
\usepackage{amsmath}%
\usepackage{amsfonts}%
\usepackage{amssymb}%
\usepackage{graphicx}%
\usepackage{mathtools}%
\usepackage{xcolor}%
\usepackage{nameref,hyperref,url}
\usepackage{xcolor}
\usepackage[foot]{amsaddr}

\calclayout

\newtheorem{theorem}{Theorem}
\theoremstyle{plain}

\newtheorem{corollary}{Corollary}

\newtheorem{definition}{Definition}

\newtheorem{lemma}{Lemma}

\newtheorem{remark}{Remark}

\numberwithin{equation}{section}

\begin{document}
\title[The Robin function for the Fractional Laplacian]{On the Robin function for the Fractional Laplacian on symmetric domains}

\author{Alejandro Ortega}
\address{Dpto. de Matem\'aticas, Universidad Carlos III de Madrid, Av. de la Universidad 30, 28911 Legan\'es (Madrid), Spain} 
\email{alortega@math.uc3m.es}

\begin{abstract} 
In this work we prove the non-degeneracy of the critical points of the Robin function for the Fractional Laplacian under symmetry and convexity assumptions on the domain $\Omega$. This work extends to the fractional setting the results of M. Grossi (cf. \cite{Grossi2002}) concerning the classical Laplace operator.
\end{abstract}
\maketitle


\section{Introduction and Main results}
Let $0<s<1$ and $\Omega\subset\mathbb{R}^N$, $N>2s$, be a smooth bounded domain. Let $G_{\Omega,t}^{s}(x)$ the Green function centered at $t\in\Omega$ of the Fractional Laplacian $(-\Delta)^s$ in $H_0^s(\Omega)$. In particular, the fractional operator we deal with is defined through the spectrum of the classical Laplace operator $-\Delta$ endowed with homogeneous Dirichlet boundary conditions on $\partial \Omega$. It is known (cf. \cite{Choi2014}) that $G_{\Omega,t}^{s}(x)$ admits the decomposition 
\begin{equation*}
G_{\Omega,t}^{s}(x)=G_{\mathbb{R}^N,t}^{s}(x)-H_{\Omega,t}^s(x),
\end{equation*}
where the singular part $G_{\mathbb{R}^N,t}^{s}(x)$ is given by the fundamental solution of $(-\Delta)^s$ in $\mathbb{R}^N$, namely
\begin{equation*}
G_{\mathbb{R}^N,t}^{s}(x)=\frac{c_{N,s}}{|x-t|^{N-2s}}\qquad \text{with}\qquad c_{N,s}=\frac{\Gamma\left(\frac{N-2s}{2}\right)}{2^{2s}\pi^{\frac{N}{2}}\Gamma(s)},
\end{equation*}
and $H_{\Omega,t}^s(x)$ is the regular part of the Green function, $H_{\Omega,t}^s(x)\in C^{\infty}(\Omega\times\Omega)$ as a function of the pair $(x,t)\in\Omega\times\Omega$ (cf. \cite[Lemma 2.4]{Choi2014}). The Robin function $\mathcal{R}_{\Omega}^s(x)$ is defined as the diagonal part of $H_{\Omega,t}^s(x)$,
\begin{equation*}
\mathcal{R}_{\Omega}^s(x)=H_{\Omega,x}^s(x)\qquad \text{for }x\in\Omega.
\end{equation*}
The Robin function plays an important role, among other fields, in the study of nonlinear elliptic problems involving the critical Sobolev exponent. In particular, for the case dealing with the classical Laplace operator, it was proved (cf. \cite{Rey1989,Han1991} see also \cite{Rey1990}) that positive solutions of nearly critical problems, $-\Delta u=u^p$, concentrate at exactly one point as the non-linearity exponent approaches the critical Sobolev exponent and this point is a critical point of the regular part of the Green function. These works answer positively to a conjecture by Brezis and Peletier about the behaviour of such solutions (cf. \cite{Brezis1989}). For the non-local case dealing with the (spectral) Fractional Laplacian an analogous concentration result was proved in \cite{Choi2014}. Hence, the study of the critical points of the Robin function turns out to be relevant in the analysis of these concentration phenomena. Then, following the approach of \cite{Grossi2002} we prove the following non-degeneracy result.
\begin{theorem}\label{Th1}
Let $\Omega\subset\mathbb{R}^N$ be a smooth bounded domain which is symmetric with respect to $x_1$ and such that 
$x_1\nu_1(x)\leq0$ for all $x\in\partial\Omega$, with $\nu(x)=(\nu_1(x),\nu_2(x),\ldots,\nu_N(x))$ the unit outward normal at a point $x\in\partial\Omega$. Then,
\begin{equation}\label{Th1a}
\frac{\partial }{\partial t_1}\mathcal{R}_{\Omega}^s(\overline{t})=0\quad \text{for } \overline{t}\in\Omega\cap\{x_1=0\},
\end{equation}
and
\begin{equation}\label{Th1b}
\frac{\partial^2 }{\partial t_1\partial t_i}\mathcal{R}_{\Omega}^s(\overline{t})= \alpha\cdot\delta_{1i}\quad \text{for } \overline{t}\in\Omega\cap\{x_1=0\},
\end{equation}
for some constant\footnote{The difference between the sign of the results of \cite{Grossi2002} and the results proven here is due to, in \cite{Grossi2002}, the normalizing constant for the fundamental solution is $\frac{1}{N(2-N)\omega_N}=-c_{N,1}$, according to our notation, which corresponds to the fundamental solution of negative Laplace operator $\Delta$. This translates in a positive sign in \cite[Eq. 2.9]{Grossi2002} in contrast to the negative sign in \eqref{representation} below, which corresponds to the fractional analogue of the classical Green's representation formula (cf. \cite[\S 2.2.4 Theorem 12]{Evans1998}).} $\alpha>0$ and $\delta_{1j}$ being the Kronecker delta.
\end{theorem}
\begin{theorem}\label{Th2}
Let $\Omega\subset\mathbb{R}^N$ be a smooth bounded domain symmetric with respect to $x_1,x_2,\ldots,x_N$ and such that $x_i\nu_i(x)\leq0$ for all $x\in\partial\Omega$, $i=1,2,\ldots,N$. Then, 
\begin{equation*}
\nabla \mathcal{R}_{\Omega}^s(0)=0
\end{equation*}
and
\begin{equation*}
\frac{\partial^2 }{\partial t_i\partial t_j}\mathcal{R}_{\Omega}^s(0)=\alpha_i\cdot\delta_{ij}
\end{equation*}
for some constants $\alpha_i>0$, $i=1,2,\ldots,N$ and $\delta_{ij}$ being the Kronecker delta.
\end{theorem}
\begin{corollary}
Let $\Omega\subset\mathbb{R}^N$ be a smooth bounded domain symmetric and convex with respect to $x_1,x_2,\ldots,x_N$. Then, the origin is a non-degenerate critical point of the Robin function $\mathcal{R}_{\Omega}^s(x)$. 
\end{corollary}
In order to avoid the difficulties arising from the non-local character of the fractional operator $(-\Delta)^s$ we make use of the extension technique introduced by Caffarelli and Silvestre (cf. \cite{Caffarelli2007} see also \cite{Cabre2010, Braendle2013, Capella2011}) that allow us to reformulate the non-local operator $(-\Delta)^s$ in terms of a directional derivative of the solution of certain auxiliary problem. Then, in Section \ref{functionalsetting} we briefly introduce the appropriate functional setting and the extension technique. In Section \ref{proofsresults} we prove the main results of the work.
\section{Functional setting}\label{functionalsetting}
The definition of powers of the positive Laplace operator $-\Delta$, in a bounded domain $\Omega$ with homogeneous Dirichlet boundary data, is carried out via the spectral decomposition, using the powers of the eigenvalues of $-\Delta$ with the same boundary condition. In particular, let $(\varphi_i,\lambda_i)$ be the eigenfunctions (normalized with respect to the $L^2(\Omega)$-norm) and eigenvalues of $-\Delta$ endowed with homogeneous Dirichlet boundary data, then $(\varphi_i,\lambda_i^s)$ are the eigenfunctions and eigenvalues of $(-\Delta)^s$ with the same boundary conditions. Therefore, the fractional operator $(-\Delta)^s$ is well defined in the space of functions
\begin{equation*}
H_0^s(\Omega)=\left\{u=\sum_{j\geq 1} a_j\varphi_j\in L^2(\Omega),\ u|_{\partial\Omega}=0:\ \|u\|_{H_0^s(\Omega)}^2= \sum_{j\geq 1}
a_j^2\lambda_j^s<\infty \right\}.
\end{equation*}
As a direct consequence of the previous definition we get
\begin{equation*}
(-\Delta)^su=\sum_{j\geq 1} a_j\lambda_j^s\varphi_j\qquad\text{and}\qquad \|u\|_{H_0^s(\Omega)}=\|(-\Delta)^{\frac{s}{2}}u\|_{L^2(\Omega)}.
\end{equation*}
This definition of the fractional powers of the Laplace operator allows us to integrate by parts in the appropriate functional space. Hence, a natural definition of energy solution to the problem 
\begin{equation}\label{probgen}
        \left\{
        \begin{tabular}{rl}
        $(-\Delta)^su=f$ &in $\Omega$, \\
        $u=0$ &on $\partial\Omega$,
        \end{tabular}
        \right.
				\tag{$P_f$}
\end{equation}
is the following.
\begin{definition}
We say that $u\in H_0^s(\Omega)$ is a solution to problem \eqref{probgen} if
\begin{equation}\label{byparts}
\int_{\Omega}(-\Delta)^{s/2}u(-\Delta)^{s/2}\psi dx=\int_{\Omega}f\psi dx,\ \ \text{for all}\ \psi\in H_0^s(\Omega).
\end{equation}
\end{definition}
In order to overcome the dificulties arising from the non-local character of the operator $(-\Delta)^s$ we use the ideas of Caffarelli and Silvestre (cf. \cite{Caffarelli2007}) together with those of \cite{Braendle2013} to give an equivalent definition of the operator $(-\Delta)^s$ defined in a bounded domain by means of an auxiliary problem.\newline
In particular, associated with the domain $\Omega$, we consider the extension cylinder $\mathcal{C}_{\Omega}=\Omega\times(0,\infty)\subset\mathbb{R}_+^{N+1}$. We denote with $(x,y)$ points that belongs to $\mathcal{C}_{\Omega}$ and with $\partial_L\Omega=\partial\Omega\times(0,\infty)$ the lateral boundary of the extension cylinder.
\begin{remark}\label{remarkoutward}
If $\Omega$ satisfies the geometric condition required in Theorem \ref{Th1} or Theorem \ref{Th2}, namely $x_i\nu_i(x)\leq0$, $i=1,2,\ldots,N$, then, by its very construction, the extension cylinder also satisfies this geometric condition. Indeed, if $\nu(x)$ is the outward normal at $x\in\partial\Omega$, then $\nu^*(x,y)=(\nu(x),0)$ is the outward normal at $(x,y)\in\partial_L\mathcal{C}_{\Omega}$, so that $x_i\nu_i^*(x,y)=x_i\nu_i(x)$.
\end{remark}
Given a function $u\in H_0^s(\Omega)$, we define its $s$-extension $w=E[u]$ to the cylinder
$\mathcal{C}_{\Omega}$ as the solution of the problem
\begin{equation*}
        \left\{
        \begin{tabular}{rl}
        $-div(y^{1-2s}\nabla w)=0\mkern+26mu$  &in $\mathcal{C}_{\Omega}$, \\
        $w(x,y)=0\mkern+26mu$   &on $\partial_L\mathcal{C}_{\Omega}$, \\
        $w(x,0)=u(x)$  & in $\Omega\times\{y=0\}$,
        \end{tabular}
        \right.
\end{equation*}
The extension function belongs to the space
\begin{equation*}
\mathcal{X}_0^s(\mathcal{C}_{\Omega})=\overline{\mathcal{C}_{0}^{\infty}\left(\mathcal{C}_{\Omega}\right)}^{\|\cdot\|_{\mathcal{X}_0^s(\mathcal{C}_{\Omega})}}\qquad\text{where}\qquad \|z\|_{\mathcal{X}_0^s(\mathcal{C}_{\Omega})}^2=\kappa_s\int_{\mathcal{C}_{\Omega}}y^{1-2s}|\nabla
z(x,y)|^2dxdy.
\end{equation*}
With that constant $\kappa_s$, whose value can be consulted in \cite{Braendle2013}, the extension operator
 $E:H_0^s(\Omega)\mapsto \mathcal{X}_0^s(\mathcal{C}_{\Omega})$ is an isometry, i.e.,
\begin{equation*}
\|E[\varphi]\|_{\mathcal{X}_0^s(\mathcal{C}_{\Omega})}=\|\varphi\|_{H_0^s(\Omega)}\qquad
\text{for all}\ \varphi\in H_0^s(\Omega).
\end{equation*}
The key point of the extension function is that it is related to the fractional Laplacian of the original function through the formula
\begin{equation*}
\frac{\partial w}{\partial \nu^s}:= -\kappa_s \lim_{y\to 0^+} y^{1-2s}\frac{\partial w}{\partial y}=(-\Delta)^su(x).
\end{equation*}
Therefore, we can reformulate the problem \eqref{probgen} in terms of the extension problem as follows
\begin{equation}\label{extension_problem}
        \left\{
        \begin{tabular}{rl}
        $-div(y^{1-2s}\nabla w)=0$  &in $\mathcal{C}_{\Omega}$, \\
        $w=0$   &on $\partial_L\mathcal{C}_{\Omega}$, \\
        $\frac{\partial w}{\partial \nu^s}=f$  & in $\Omega\times\{y=0\}$.
        \end{tabular}
        \right.
        \tag{$P_f^*$}
\end{equation}
An energy solution of this problem is a function $w\in \mathcal{X}_0^s(\mathcal{C}_{\Omega})$ such that
\begin{equation}\label{sol_ext}
\kappa_s\int_{\mathcal{C}_{\Omega}} y^{1-2s}\nabla w\nabla\varphi dxdy=\int_{\Omega} f\varphi(x,0)dx\qquad \forall \varphi\in\mathcal{X}_0^s(\mathcal{C}_{\Omega}).
\end{equation}
Given $w\in\mathcal{X}_0^s(\mathcal{C}_{\Omega})$ a solution to \eqref{extension_problem} the function
$u(x)=Tr[w]=w(x,0)$ belongs to the space $H_0^s(\Omega)$ and it is an energy solution to \eqref{probgen} and vice versa, if $u\in H_0^s(\Omega)$ is a solution to \eqref{probgen} then $w=E[u]\in \mathcal{X}_0^s(\mathcal{C}_{\Omega})$ is a solution to \eqref{extension_problem} and, as a consequence, both
formulations are equivalent.\newline
In particular, for the Green function $G_{\Omega,t}^s(x)$ centered at $t\in\Omega$, i.e., the solution to
\begin{equation}\label{ggren}
        \left\{
        \begin{tabular}{rl}
        $(-\Delta)^sG_{\Omega,t}^s(x)=\delta_{t}$  &in $\Omega$, \\
        $G_{\Omega,t}^s(x)=0\mkern+5mu$  & on $\partial\Omega$,
        \end{tabular}
        \right.
\end{equation}
its $s$-extension, $E[G_{\Omega,t}^s](x,y)$, is given by the solution to the problem
\begin{equation}\label{extension_problem_green}
        \left\{
        \begin{tabular}{rl}
        $-div(y^{1-2s}\nabla E[G_{\Omega,t}^s])=0\mkern+5.5mu$  &in $\mathcal{C}_{\Omega}$, \\
        $E[G_{\Omega,t}^s]=0\mkern+5.5mu$  &on $\partial_L\mathcal{C}_{\Omega}$, \\
        $\frac{\partial }{\partial \nu^s}E[G_{\Omega,t}^s]=\delta_{t}$  & in $\Omega\times\{y=0\}$.
        \end{tabular}
        \right.
\end{equation}
Moreover, if $\displaystyle G_{\mathbb{R}^N,t}^{s}(x)=\frac{c_{N,s}}{|x-t|^{N-2s}}$, its $s$-extension is given by $\displaystyle E[G_{\mathbb{R}^N,t}^s](x,y)=\frac{c_{N,s}}{|(x-t,y)|^{N-2s}}$.

\section{Proof of main results}\label{proofsresults}
Let us assume that $\Omega$ is a smooth bounded domain symmetric with respect to the hyperplane $\{x_1=0\}$.
\begin{lemma}\label{lemmasymetric} For $\overline{t}\in\Omega\cap\{x_1=0\}$, $(x,y)\in\mathcal{C}_{\Omega}$, $x=(x_1,x')$ with $x_1\in\mathbb{R}$ and $x'\in\mathbb{R}^{N-1}$ we have
\begin{equation*}
E[G_{\Omega,\overline{t}}^s]\big((x_1,x'),y\big)=E[G_{\Omega,\overline{t}}^s]\big((-x_1,x'),y\big).
\end{equation*}
\end{lemma}

\begin{proof}
Because of \eqref{extension_problem_green} and \eqref{sol_ext}, for any $\varphi\in C_0^{\infty}(\mathcal{C}_{\Omega})$, we have
\begin{equation*}
\kappa_s\int_{\mathcal{C}_{\Omega}}y^{1-2s}\nabla E[G_{\Omega,\overline{t}}^s](x,y)\nabla\varphi(x,y)dydx=\int_{\Omega}\delta_{\overline{t}}\varphi(x,0)dx=\varphi(\overline{t},0).
\end{equation*}
Since $\Omega$ is symmetric with respect to $\{x_1=0\}$ the extension cylinder $\mathcal{C}_{\Omega}$ is also symmetric with respect to $\{x_1=0\}$, thus we have
\begin{equation*}
\begin{split}
\kappa_s\int_{\mathcal{C}_{\Omega}}y^{1-2s}\nabla E[G_{\Omega,\overline{t}}^s]\big((-x_1,x'),y\big)\nabla\varphi(x,y)dydx=&\kappa_s\int_{\mathcal{C}_{\Omega}}y^{1-2s}\nabla E[G_{\Omega,\overline{t}}^s](x,y)\nabla\varphi\big((-x_1,x'),y\big)dydx\\
=&\int_{\Omega}\delta_{\overline{t}}\varphi\big((-x_1,x'),0\big)dx\\
=&\varphi(\overline{t},0),
\end{split}
\end{equation*}
because of $\overline{t}$ belongs to the hyperplane $\{x_1=0\}$ so that $\overline{t}=(0,\overline{t}')$. Hence
\begin{equation*}
\int_{\mathcal{C}_{\Omega}}y^{1-2s}\left(\nabla E[G_{\Omega,\overline{t}}^s](x,y)-\nabla E[G_{\Omega,\overline{t}}^s]\big((-x_1,x'),y\big)\right)\nabla\varphi(x,y)dydx=0\quad\forall\varphi\in C_0^{\infty}(\mathcal{C}_{\Omega}),
\end{equation*}
and the claim follows.
\end{proof}
\begin{remark}
Arguing as in Lemma \ref{lemmasymetric} and using \eqref{ggren} and \eqref{byparts} it follows that
\begin{equation*}
\int_{\Omega}\left((-\Delta)^{s/2}G_{\Omega,\overline{t}}^s(x)-(-\Delta)^{s/2}G_{\Omega,\overline{t}}^s(-x_1,x')\right)(-\Delta)^{s/2}\psi(x)dx=0\qquad\forall\psi\in C_0^{\infty}(\Omega),
\end{equation*}
and, thus, $G_{\Omega,\overline{t}}^s(x_1,x')=G_{\Omega,\overline{t}}^s(-x_1,x')$ for $\overline{t}\in\Omega\cap\{x_1=0\}$, from where we can also deduce Lemma \ref{lemmasymetric}.
\end{remark}
\begin{lemma}\label{lemmaodd}
Fixed $\overline{t}\in\Omega\cap\{x_1=0\}$, let $U_1(x,y)$ be the solution to the problem 
\begin{equation}\label{pu1}
        \left\{
        \begin{tabular}{rl}
        $-div(y^{1-2s}\nabla U_1)=0\mkern+78.4mu$  &in $\mathcal{C}_{\Omega}$, \\
        $U_1(x,y)=\frac{\partial}{\partial x_1}E[G_{\Omega,\overline{t}}^s]$  &on $\partial_L\mathcal{C}_{\Omega}$, \\
        $\frac{\partial }{\partial \nu^s}U_1(x,0)=0\mkern+78.4mu$  & in $\Omega\times\{y=0\}$.
        \end{tabular}
        \right.			
\end{equation}
and assume that $x_1\nu_1(x)\leq0$ for all $x\in\partial\Omega$. Then 
\begin{equation*}
\frac{\partial}{\partial x_1}U_1(x,0)\Big|_{x=\overline{t}}<0\qquad\text{and}\qquad \frac{\partial}{\partial x_i}U_1(x,0)\Big|_{x=\overline{t}}=0\quad\text{for } i=2,\ldots,N.
\end{equation*}

\end{lemma}
\begin{proof}
Taking in mind \eqref{pu1}, \eqref{extension_problem_green} and using the Green's identities, we have the representation formula 
\begin{equation}\label{representation}
U_1(t,0)=-\kappa_s\int_{\partial_L\mathcal{C}_{\Omega}} y^{1-2s}\frac{\partial}{\partial x_1}E[G_{\Omega,\overline{t}}^s](x,y)\frac{\partial }{\partial\nu_{(x,y)}^*}E[G_{\Omega,t}^s](x,y)d\sigma_{(x,y)}.
\end{equation}
Because of Lemma \ref{lemmasymetric}, the function $\frac{\partial}{\partial x_1}E[G_{\Omega,\overline{t}}^s](x,y)$ is odd in the $x_1$ variable, namely,
\begin{equation*}
\frac{\partial}{\partial x_1}E[G_{\Omega,\overline{t}}^s](x_1,x',y)=-\frac{\partial}{\partial x_1}E[G_{\Omega,\overline{t}}^s](-x_1,x',y).
\end{equation*}
Then, by \eqref{representation}, the function $U_1(\cdot,0)$ is odd in the $e_1$ direction and, thus, $U_1(\overline{t},0)=0$ for $ \overline{t}\in\{x_1=0\}$. As a consequence, $\frac{\partial}{\partial x_i}U_1(x,0)\Big|_{x=\overline{t}}=0$ for $i=2,\ldots,N$. Next, since $E[G_{\Omega,t}^s](x,y)=0$ for $(x,y)\in\partial_L\mathcal{C}_{\Omega}$, so that 
\begin{equation*}
\nu^*(x,y)=-\frac{\nabla E[G_{\Omega,t}^s]}{|E[G_{\Omega,t}^s]|}\quad\text{at } (x,y)\in\partial_L\mathcal{C}_{\Omega},
\end{equation*}
and $E[G_{\Omega,t}^s](x,y)>0$ for $(x,y)\in\mathcal{C}_{\Omega}$, the hypotheses $x_1\nu_1(x)\leq0$ for all $x\in\partial\Omega$ implies (see Remark \ref{remarkoutward}) that 
\begin{equation*}
\frac{\partial}{\partial x_1}E[G_{\Omega,t}^s](x,y)\geq0\text{ for } \{x_1<0\}\cap\partial_L\mathcal{C}_{\Omega}\quad\text{and}\quad \frac{\partial}{\partial x_1}E[G_{\Omega,t}^s](x,y)\leq0\text{ for }\{x_1>0\}\cap\partial_L\mathcal{C}_{\Omega}.
\end{equation*}
Then, the Maximum Principle gives $U_1(x,0)>0$ for $x_1<0$ and, applying the Hopf Lemma to $U_1(x,0)$ in $\mathcal{C}_{\Omega}\cap\{x_1<0\}$ we conclude $\frac{\partial}{\partial x_1}U_1(x,0)\Big|_{x=\overline{t}}<0$.
\end{proof}
Next Lemma extends to the fractional setting \cite[Theorem 4.4]{Brezis1989} and \cite[Lemma 2.3]{Grossi2002}.
\begin{lemma}\label{gradrobin} For any $t\in\Omega$ we have
\begin{equation*}
\nabla\mathcal{R}_{\Omega}^s(t)=\kappa_s\int_{\partial_L\mathcal{C}_{\Omega}}y^{1-2s}\left(\frac{\partial E[G_{\Omega,t}^s]}{\partial\nu_{(x,y)}^*}(x,y)\right)^2\nu(x) d\sigma_{(x,y)},
\end{equation*}
that is,
\begin{equation}\label{grad}
\frac{\partial}{\partial t_i}\mathcal{R}_{\Omega}^s(t)=\kappa_s\int_{\partial_L\mathcal{C}_{\Omega}}y^{1-2s}\left(\frac{\partial E[G_{\Omega,t}^s]}{\partial\nu_{(x,y)}^*}(x,y)\right)^2\nu_i(x) d\sigma_{(x,y)},
\end{equation}
and
\begin{equation}\label{grad2}
\frac{\partial^2}{\partial t_i\partial t_j}\mathcal{R}_{\Omega}^s(t)=2\kappa_s\int_{\partial_L\mathcal{C}_{\Omega}}y^{1-2s}\frac{\partial E[G_{\Omega,x}^s]}{\partial t_i}(t,y)\frac{\partial}{\partial t_j}\left(\frac{\partial E[G_{\Omega,t}^s]}{\partial\nu_{(x,y)}^*}(x,y)\right) d\sigma_{(x,y)}.
\end{equation}
\end{lemma}

\begin{proof}
We begin by proving \eqref{grad}. Let $u\in H_0^s(\Omega)$ and consider its extension problem
\begin{equation*}
        \left\{
        \begin{tabular}{rl}
        $-div(y^{1-2s}\nabla E[u])=0\mkern+52mu$  &in $\mathcal{C}_{\Omega}$, \\
        $E[u]=0\mkern+52mu$   &on $\partial_L\mathcal{C}_{\Omega}$, \\
        $\frac{\partial}{\partial\nu^s}E[u]=(-\Delta)^su$  & in $\Omega\times\{y=0\}$.
        \end{tabular}
        \right.
\end{equation*}
Multiplying by $\varphi(x,y)$ and integrating by parts we get
\begin{equation}\label{eq1}
\begin{split}
\kappa_s\int_{\mathcal{C}_{\Omega}}y^{1-2s}\nabla E[u](x,y) \nabla\varphi(x,y) dydx=&\int_{\Omega}(-\Delta)^su(x)\varphi(x,0) dx\\
&+\kappa_s\int_{\partial_L\mathcal{C}_{\Omega}}y^{1-2s}\varphi(x,y)\frac{\partial E[u]}{\partial \nu_{(x,y)}^*}(x,y)d\sigma_{(x,y)}.
\end{split}
\end{equation}
Let us set $\varphi=\frac{\partial}{\partial x_i}E[u]=e_i\cdot\nabla E[u]$ with $e_i$ the unitary vector along the $x_i\,$-axis. Then,
\begin{equation}\label{eq1b}
\begin{split}
\kappa_s\int_{\mathcal{C}_{\Omega}}y^{1-2s}\nabla E[u] \nabla\varphi dydx=&\kappa_s\int_{\mathcal{C}_{\Omega}}y^{1-2s}e_i\cdot\nabla\left(\frac{|\nabla E[u]|^2}{2}\right)dydx\\
=&\frac{\kappa_s}{2}\int_{\partial_L\mathcal{C}_{\Omega}}y^{1-2s}|\nabla E[u]|^2 \big(e_i\cdot\nu^*(x,y)\big) d\sigma_{(x,y)}\\
&-\frac{\kappa_s}{2}\int_{\mathcal{C}_{\Omega}}div(y^{1-2s}e_i)|\nabla E[u]|^2dydx\\
=&\frac{\kappa_s}{2}\int_{\partial_L\mathcal{C}_{\Omega}}y^{1-2s}\left(\frac{\partial E[u]}{\partial\nu_{(x,y)}^*}\right)^2\nu_i(x) d\sigma_{(x,y)}
\end{split}
\end{equation}
because of $E[u]=0$ on $\partial_L\mathcal{C}_{\Omega}$, so that $\frac{\partial E[u]}{\partial\nu_{(x,y)}^*}=-|\nabla E[u]|$, and $\nu^*(x,y)=(\nu(x),0)$. On the other hand, since $\varphi(x,0)=\frac{\partial}{\partial x_i}E[u](x,0)=\frac{\partial}{\partial x_i}u(x)$ and 
\begin{equation}\label{eq1c}
\varphi(x,y)\frac{\partial E[u]}{\partial\nu_{(x,y)}^*}(x,y)=|\nabla E[u](x,y)|^2(e_i\cdot\nu^*(x,y))=\left(\frac{\partial E[u]}{\partial\nu_{(x,y)}^*}(x,y)\right)^2(e_i\cdot\nu^*(x,y)),
\end{equation}
from \eqref{eq1}, \eqref{eq1b} and \eqref{eq1c}, we get 
\begin{equation}\label{eq2}
-\frac{\kappa_s}{2} \int_{\partial_L\mathcal{C}_{\Omega}}y^{1-2s}\left(\frac{\partial E[u]}{\partial\nu_{(x,y)}^*}(x,y)\right)^2\nu_i(x) d\sigma_{(x,y)}=\int_{\Omega}\frac{\partial}{\partial x_i}u(x)(-\Delta)^su(x)dx.
\end{equation}
Next, let $u=u_{\rho}$ be the solution to the linear problem 
\begin{equation*}
        \left\{
        \begin{tabular}{rl}
        $(-\Delta)^su=\eta_{\rho}(t)$  &in $\Omega$, \\
        $u=0\mkern+30mu$  & on $\partial\Omega$,
        \end{tabular}
        \right.			
\end{equation*}
where $\eta_{\rho}(t)=\frac{1}{|B_{\rho}(t)|}\chi_{B_{\rho}(t)}$ with $|B_{\rho}(t)|=\frac{\pi^{\frac{N}{2}}}{\Gamma\left(\frac{N}{2}+1\right)}\rho^N$ the volume of the n-dimensional ball of radius $\rho>0$ centered at $t$ and $\chi_A$ the characteristic function of the set $A$. The functions $\eta_{\rho}(t)$ converge weakly to the Dirac delta $\delta_{t}$ and
\begin{equation*}
u_{\rho}(x)\to G_{\Omega,t}^s(x)\quad\text{as }\rho\to0^+.
\end{equation*}
Also, let $v=v_{\rho}$ be the solution to the linear problem 
\begin{equation*}
     (-\Delta)^sv=\eta_{\rho}(t)\qquad \mbox{in }\mathbb{R}^N,
\end{equation*} 
such that $\lim\limits_{|x|\to+\infty}v(x)=0$. The function $v_{\rho}$ is symmetric with respect to the point $t\in\Omega$ and
\begin{equation*}
v_{\rho}(x)\to G_{\mathbb{R}^N,t}^s(x)\quad\text{as }\rho\to0^+.
\end{equation*}
Therefore,
\begin{equation*}
\begin{split}
\int_{\Omega}\frac{\partial}{\partial x_i}u(x)(-\Delta)^su(x)dx&=\int_{\Omega}\frac{\partial}{\partial x_i}u_{\rho}(x)\eta_{\rho}(t)dx\\
&=\int_{\Omega}\frac{\partial}{\partial x_i}(u_{\rho}(x)-v_{\rho}(x))\eta_{\rho}(t)dx+\int_{\Omega}\frac{\partial}{\partial x_i}v_{\rho}(x)\eta_{\rho}(t)dx\\
&=\int_{\Omega}\frac{\partial}{\partial x_i}(u_{\rho}(x)-v_{\rho}(x))\eta_{\rho}(t)dx,
\end{split}
\end{equation*}
by the symmetry of the function $v_{\rho}$. Hence, \eqref{eq2} give us
\begin{equation*}
-\frac{\kappa_s}{2} \int_{\partial_L\mathcal{C}_{\Omega}}y^{1-2s}\left(\frac{\partial E[u_{\rho}]}{\partial\nu_{(x,y)}^*}(x,y)\right)^2\nu_i(x) d\sigma_{(x,y)}=\int_{\Omega}\frac{\partial}{\partial x_i}(u_{\rho}(x)-v_{\rho}(x))\eta_{\rho}(t)dx.
\end{equation*}
Taking $\rho\to0^+$, so that $\displaystyle u_{\rho}(x)-v_{\rho}(x)\to G_{\Omega,t}^s(x)-G_{\mathbb{R}^N,t}^s(x)=-H_{\Omega,t}^s(x)$, we get
\begin{equation*}
\int_{\Omega}\frac{\partial}{\partial x_i}(u_{\rho}(x)-v_{\rho}(x))\eta_{\rho}(t)dx=-\frac{\partial}{\partial x_i}H_{\Omega,t}^s(x){\Big|}_{x=t}.
\end{equation*}
Therefore,
\begin{equation*}
\frac{\kappa_s}{2} \int_{\partial_L\mathcal{C}_{\Omega}}y^{1-2s}\left(\frac{\partial E[G_{\Omega,t}^s]}{\partial\nu_{(x,y)}^*}(x,y)\right)^2\nu_i(x) d\sigma_{(x,y)}=\frac{\partial}{\partial x_i}H_{\Omega,t}^s(x){\Big|}_{x=t}.
\end{equation*}
Finally, we differentiate the relation $\mathcal{R}_{\Omega}^s(x)=H_{\Omega,x}^s(x)$, so that 
\begin{equation*}
\frac{\partial}{\partial x_i}\mathcal{R}_{\Omega}^s(x)\Big|_{x=t}=\frac{\partial}{\partial x_i}H_{\Omega,t}^s(x)\Big|_{x=t}+\frac{\partial}{\partial t_i}H_{\Omega,t}^s(x)\Big|_{x=t}=2\frac{\partial}{\partial x_i}H_{\Omega,t}^s(x)\Big|_{x=t},
\end{equation*}
since $H_{\Omega,t}^s(x)=H_{\Omega,x}^s(t)$. Thus, we have
\begin{equation*}
\kappa_s\int_{\partial_L\mathcal{C}_{\Omega}}y^{1-2s}\left(\frac{\partial E[G_{\Omega,t}^s]}{\partial\nu_{(x,y)}^*}(x,y)\right)^2\nu_i(x) d\sigma_{(x,y)}=\frac{\partial}{\partial x_i}\mathcal{R}_{\Omega}^s(x)\Big|_{x=t}.
\end{equation*}
After renaming variables we get \eqref{grad} and, repeating the steps above for $i=1,2,\ldots,N$, we conclude
\begin{equation*}
\kappa_s\int_{\partial_L\mathcal{C}_{\Omega}}y^{1-2s}\left(\frac{\partial E[G_{\Omega,t}^s]}{\partial\nu_{(x,y)}^*}(x,y)\right)^2\nu(x) d\sigma_{(x,y)}=\nabla\mathcal{R}_{\Omega}^s(t),
\end{equation*}
with $\nu(x)=(\nu_1(x),\nu_2(x),\ldots,\nu_N(x))$. To prove \eqref{grad2} we derive \eqref{grad} with respect to $t_j$ and we get
\begin{equation*}
\frac{\partial^2}{\partial t_i\partial t_j}\mathcal{R}_{\Omega}^s(t)=2\kappa_s\int_{\partial_L\mathcal{C}_{\Omega}}y^{1-2s}\frac{\partial E[G_{\Omega,t}^s]}{\partial\nu_{(x,y)}^*}(x,y)\nu_i(x)\frac{\partial}{\partial t_j}\left(\frac{\partial E[G_{\Omega,t}^s]}{\partial\nu_{(x,y)}^*}(x,y)\right) d\sigma_{(x,y)}.
\end{equation*}
Since $E[G_{\Omega,t}^s](x,y)=0$ on $\partial_L\mathcal{C}_{\Omega}$ and $E[G_{\Omega,t}^s](x,y)=E[G_{\Omega,x}^s](t,y)$ (because of $G_{\Omega,t}^s(x)=G_{\Omega,x}^s(t)$) we have 
\begin{equation}\label{partiali}
\frac{\partial E[G_{\Omega,t}^s]}{\partial\nu_{(x,y)}^*}(x,y)\nu_i(x)=\frac{\partial E[G_{\Omega,t}^s]}{\partial x_i}(x,y)=\frac{\partial E[G_{\Omega,x}^s]}{\partial t_i}(t,y)
\end{equation}
and \eqref{grad2} follows.
\end{proof}

\begin{proof}[Proof of Theorem \ref{Th1}]
Let $U_1$ be the solution of \eqref{pu1} and $\overline{t}\in\Omega\cap\{x_1=0\}$. Since $U_1(x,0)$ is odd with respect to $x_1$ (see the proof of Lemma \ref{lemmaodd}), so that $U_1(\overline{t},0)=0$, by \eqref{representation}, \eqref{partiali} with $i=1$ and \eqref{grad} we get
\begin{equation}\label{repre2}
\begin{split}
0=U_1(\overline{t},0)&=-\kappa_s\int_{\partial_L\mathcal{C}_{\Omega}} y^{1-2s}\frac{\partial}{\partial x_1}E[G_{\Omega,\overline{t}}^s](x,y)\frac{\partial }{\partial\nu_{(x,y)}^*}E[G_{\Omega,\overline{t}}^s](x,y)d\sigma_{(x,y)}\\
&=-\kappa_s\int_{\partial_L\mathcal{C}_{\Omega}}y^{1-2s}\left(\frac{\partial E[G_{\Omega,\overline{t}}^s]}{\partial\nu_{(x,y)}^*}(x,y)\right)^2\nu_1(x) d\sigma_{(x,y)}\\
&=-\frac{\partial}{\partial t_i}\mathcal{R}_{\Omega}^s(\overline{t})
\end{split}
\end{equation}
and we conclude \eqref{Th1a}.
Differentiating \eqref{representation} with respect to $t_i$ and using that $E[G_{\Omega,z}^s](x,y)=E[G_{\Omega,x}^s](z,y)$ together with \eqref{grad2} we get
\begin{equation*}
\begin{split}
\frac{\partial}{\partial t_i}U_1(\overline{t},0)&=-\kappa_s\int_{\partial_L\mathcal{C}_{\Omega}} y^{1-2s}\frac{\partial}{\partial x_1}E[G_{\Omega,\overline{t}}^s](x,y)\frac{\partial}{\partial t_i}\left(\frac{\partial }{\partial\nu_{(x,y)}^*}E[G_{\Omega,\overline{t}}^s](x,y)\right)d\sigma_{(x,y)}\\
&=-\kappa_s\int_{\partial_L\mathcal{C}_{\Omega}} y^{1-2s}\frac{\partial}{\partial t_1}E[G_{\Omega,x}^s](\overline{t},y)\frac{\partial}{\partial t_i}\left(\frac{\partial }{\partial\nu_{(x,y)}^*}E[G_{\Omega,\overline{t}}^s](x,y)\right)d\sigma_{(x,y)}\\
&=-\frac{1}{2}\frac{\partial^2}{\partial t_1\partial t_i}\mathcal{R}_{\Omega}^s(\overline{t}).
\end{split}
\end{equation*}
From Lemma \ref{lemmaodd} we conclude \eqref{Th1b}.
\end{proof}

\begin{proof}[Proof of Theorem \ref{Th2}]
Repeating the proof of Lemma \ref{lemmasymetric} assuming that $\Omega$ is a smooth bounded domain symmetric with respect to the hyperplanes $\{x_i=0\}$ for all $i=1,2,\ldots,N$ we get that $E[G_{\Omega,0}^s](x,y)$ is symmetric, in the $x$-variable, with respect to $x=0$. Using this and repeating, for $i=1,2,\ldots,N$, the proof of Lemma \ref{lemmaodd} with $U_i(x,y)$ being the solution to the problem 
\begin{equation*}
        \left\{
        \begin{tabular}{rl}
        $-div(y^{1-2s}\nabla U_i)=0\mkern+78.4mu$  &in $\mathcal{C}_{\Omega}$, \\
        $U_i(x,y)=\frac{\partial}{\partial x_i}E[G_{\Omega,0}^s]$  &on $\partial_L\mathcal{C}_{\Omega}$, \\
        $\frac{\partial }{\partial \nu^s}U_i(x,0)=0\mkern+78.4mu$  & in $\Omega\times\{y=0\}$.
        \end{tabular}
        \right.			
\end{equation*}
under the hypotheses $x_i\nu_i(x)\leq0$ for all $x\in\partial\Omega$, we get
\begin{equation*}
\frac{\partial}{\partial x_i}U_i(0,0)<0\qquad\text{and}\qquad \frac{\partial}{\partial x_j}U_i(0,0)=0\quad\text{for } j\neq i.
\end{equation*} 
Combining this with Lemma \ref{gradrobin} it follows that the hessian matrix of the Robin function computed at
zero is diagonal.
\end{proof}



\end{document}